\documentclass[12pt,a4paper, fleqn] {amsart}
\usepackage{amsmath,amscd,amsfonts,eucal,latexsym,amssymb,mathrsfs}
\usepackage{graphicx}
\newtheorem{teorema}{Theorem}[section]
\newtheorem{lemma}[teorema]{Lemma}

\theoremstyle{definition}

\theoremstyle{remark}

\theoremstyle{assioma}

\def\cd{\mathcal{D}}
\def\caa {\mathcal {A}}
\def\cbb {\mathcal {B}}
\def\wtcd{ \widetilde{\mathcal{D}}}
\def\wpcd{\widetilde{\mathcal{D}}'}
\font\picco=cmr10 scaled 850
\font\pio=cmr10 scaled 1230
\font\pioa=cmb10 scaled 1800
\def\veps{\varepsilon}
\def\cc{\tilde c}
\def\ccg{\mathcal G}
\def\sskip{\bigskip\bigskip\bigskip\bigskip\bigskip}
\def\ccb{{\mathcal B}}
\def\td{\widetilde{\mathcal D}}
\def\cd{\mathcal D}
\def\cdni#1{\cd^{(#1)}_{\mathcal N}}
\def\cdn{\cd_{\mathcal N}}
\def\tdn{\td_{\mathcal N}}
 \def\cdi#1{\cd^{(#1)}}
\def\tqo{{\tq}^o}
\def\tqb{{\tq}^b}
\def\cqz{{\cq}^z}
\def\cq{ {\mathcal Q}}
\def\tq{\widetilde {\mathcal Q}}
\def\wq{\widetilde   q}
\def\ccn{{\mathcal N}}
\def\cv{{\mathcal U}}
\def\bte{{\bf \Theta}}
\def\vpi{\varpi}
\def\cc{\tilde c}
\def\cn{{\mathcal N}}
\def\cer#1{\mathrel{\smash{\mathop{#1} \limits^{\circ} }}}
\def\enpr{\quad \vrule height .9ex width .8ex depth -.1ex}
\def\rr{{\, \mathcal R\,}}
\def\bar{\overline}
\def\nont{ \rlap {\ /} \Theta }
\def\bre{\mathbb {R}}
\def\bna{\mathbb N}
\def\bnaso{\mathbb N\setminus\{0\}}
\def\bin{\Bbb Z}
\def\bra{\Bbb Q}
\def\nor#1{||#1||_{\infty}}
\def\nop#1{|||#1|||}
\def\ac{\`}
\def\ze{\zeta}
\def\llim{\mathop{\longrightarrow}}
\def\ssim{\mathop{\sim}}
\def\Osc{\mathop{\text{Osc}}}
\def\vo{V^{(0)}}
\def\von{V^{(1)}}
\def\vn#1{V^{(#1)}}
\def\rvo{\bre^{V^{(0)}}}
\def\rvon{\bre^{V^{(1)}}}
\def\lb{\lambda}
\def\Lb{\Lambda}
\def\Lhb{\widehat\Lambda}
\def\Lbs{\Lambda^*}
\def\Lbt{\widetilde{\Lambda^*}}
\def\Lbn{\Lambda^*_{\infty}}
\def\Lbnt{\widetilde{\Lambda^*}_{\infty}}
\def\Av{(No)}
\def\cl{{\mathcal L}}
\def\cf{{\mathcal F}}
\def\pch{(Ex)}
\def\ppp{\Pi}
\def\ppm{\Pi^{-1}}
\def\wtc{\widetilde C}
\def\ta{\tilde A}
\def\s,{\quad $\,$}
\def\wt{\widetilde \cv}
\def\brr{\bar R}
\def\hrr{\widehat R}
\def\brs{\witi {\bar  R}}
\def\Ga{\Gamma}
\def\ga{\bar I}
\def\bi{\bar i}
\def\bp{\bar p}
\def\rta{\tilde r_A}
\def\kde{\witi K}
\font\pio=cmr10 scaled 1230
\def\gr{\mathcal Gr}
\def\qua{\quad$\,$}
\def\smad{\smallskip Proof.\ }
\def\sskip{\bigskip\bigskip\bigskip\bigskip\bigskip}
\def\s,{\quad $\,$}
\def\btt{\tilde R}
\def\witi{\widetilde}
\def\gai#1{{#1_{\bar I}}}
\def\wv{\widetilde V}
\def\disp{\displaystyle}
\def\rge{r_{(g,\eta,n)}}
\bigskip

\begin{document}

\title{
 Fixed Points of anti-attracting maps 
 and Eigenforms on Fractals}
 \maketitle
 \author{Roberto Peirone \footnote{
  : e-mail: {\sf peirone@axp.mat.uniroma2.it}, Phone: +39\,0672594610,
    Fax: +39\,0672594699}},
  {Universit\`a di Roma-Tor Vergata,
   Dipartimento di Matematica,        
Via della Ricerca Scientifica, 00133 Roma, Italy}


\def\cd{\mathcal{D}}
\def\caa{ {\mathcal {A}}}
\def\wtcd{ \widetilde{\mathcal{D}}}
\def\wpcd{\widetilde{\mathcal{D}}'}
\font\picco=cmr10 scaled 850
\font\pio=cmr10 scaled 1230
\font\pioa=cmb10 scaled 1800
\def\veps{\varepsilon}
\def\ccg{\mathcal G}
\def\sskip{\bigskip\bigskip\bigskip\bigskip\bigskip}
\def\ccb{{\mathcal B}}
\def\td{\widetilde{\mathcal D}}
\def\cd{\mathcal D}
\def\cdni#1{\cd^{(#1)}_{\mathcal N}}
\def\cdn{\cd_{\mathcal N}}
\def\tdn{\td_{\mathcal N}}
 \def\cdi#1{\cd^{(#1)}}
\def\tqo{{\tq}^o}
\def\tqb{{\tq}^b}
\def\cqz{{\cq}^z}
\def\cq{ {\mathcal Q}}
\def\tq{\widetilde {\mathcal Q}}
\def\wq{\widetilde   q}
\def\ccn{{\mathcal N}}
\def\cv{{\mathcal U}}
\def\bte{{\bf \Theta}}
\def\vpi{\varpi}
\def\cc{\tilde c}
\def\cn{{\mathcal N}}
\def\cer#1{\mathrel{\smash{\mathop{#1} \limits^{\circ} }}}
\def\enpr{\quad \vrule height .9ex width .8ex depth -.1ex}
\def\rr{{\, \mathcal R\,}}
\def\bar{\overline}
\def\nont{ \rlap {\ /} \Theta }
\def\bre{\mathbb {R}}
\def\bna{\mathbb N}
\def\bnaso{\mathbb N\setminus\{0\}}
\def\bin{\Bbb Z}
\def\bra{\Bbb Q}
\def\nor#1{||#1||_{\infty}}
\def\nop#1{|||#1|||}
\def\ac{\`}
\def\ze{\zeta}
\def\llim{\mathop{\longrightarrow}}
\def\ssim{\mathop{\sim}}
\def\Osc{\mathop{\text{Osc}}}
\def\vo{V^{(0)}}
\def\von{V^{(1)}}
\def\vn#1{V^{(#1)}}
\def\rvo{\bre^{V^{(0)}}}
\def\rvon{\bre^{V^{(1)}}}
\def\lb{\lambda}
\def\Lb{\Lambda}
\def\Lhb{\widehat\Lambda}
\def\Lbs{\Lambda^*}
\def\Lbt{\widetilde{\Lambda^*}}
\def\Lbn{\Lambda^*_{\infty}}
\def\Lbnt{\widetilde{\Lambda^*}_{\infty}}
\def\Av{(No)}
\def\cl{{\mathcal L}}
\def\cf{{\mathcal F}}
\def\pch{(Ex)}
\def\ppp{\Pi}
\def\ppm{\Pi^{-1}}
\def\wtc{\widetilde C}
\def\ta{\tilde A}
\def\s,{\quad $\,$}
\def\wt{\widetilde \cv}
\def\brr{\bar R}
\def\hrr{\widehat R}
\def\brs{\witi {\bar  R}}
\def\Ga{\Gamma}
\def\ga{\bar I}
\def\bi{\bar i}
\def\bp{\bar p}
\def\rta{\tilde r_A}
\def\kde{\witi K}
\font\pio=cmr10 scaled 1230
\def\gr{\mathcal Gr}
\def\qua{\quad$\,$}
\def\smad{\smallskip Proof.\ }
\def\sskip{\bigskip\bigskip\bigskip\bigskip\bigskip}
\def\s,{\quad $\,$}
\def\btt{\tilde R}
\def\witi{\widetilde}
\def\gai#1{{#1_{\bar I}}}
\def\wv{\widetilde V}
\def\disp{\displaystyle}
\def\rge{r_{(g,\eta,n)}}
\def\innt{\text{\rm int}}
\bigskip


\begin{abstract}
An important problem in analysis on fractals is the
 existence   of a self-similar energy
on   finitely ramified fractals.  
The self-similar energies are constructed in terms
of eigenforms, that is, eigenvectors of a special
nonlinear operator. Previous results by C. Sabot
and V. Metz give conditions for the existence of
an eigenform. In this paper, I give a different and
probably  shorter proof
of the previous results, 
which appears to be suitable for improvements.
Such a proof is based on a fixed-point theorem for anti-attracting maps on a 
convex set. 
 \end{abstract}

\section{Introduction}
The subject of this paper is analysis on fractals.
Much of analysis on fractals 
is based  on an energy on them.
 Therefore, an important problem is
 the construction of   self-similar
  Dirichlet forms on fractals, i.e. energies.
In this paper, we investigate the finitely ramified fractals. 
This means more or less
 that the intersection of each pair of copies
of the fractal is a finite set. The Sierpinski Gasket and its generalizations, the 
Vicsek Set and the Lindstr\o m Snowflake are finitely 
ramified fractals, while the Sierpinski Carpet is not.
 The class of P.C.F. self-similar sets was introduced by Kigami in \cite{Kigp} 
 and a general theory with many examples can be found in \cite{Kigb}. 
 In this paper, we consider a subclass of
 the class  P.C.F. self-similar sets,
 with a very mild additional requirement, which is described in Section 2.
 This is the same setting as in other papers of mine
e.g.,  \cite{Pe1} and is essentially the same setting as in \cite{HHW},
and in other papers 
 (\cite{Str}, \cite{Sa}). 
We require that  every point in the initial set is a fixed point 
 of one of the contractions defining the fractal.
 Moreover, we require that the fractal is connected.

On such a class of fractals,
 the basic tool used to construct a Dirichlet form
 is a self-similar \textit{discrete}
 Dirichlet form defined on a special
 finite subset $V^{(0)}$ of the fractal. This subset is a sort 
 of boundary of the fractal. Such self-similar Dirichlet forms 
 are the \textit{eigenforms}, i.e., the eigenvectors of a special
  nonlinear operator $\Lambda_r$ called \textit{renormalization operator}, 
  which depends on a set of positive \textit{weights} $r_i$ placed on the cells of the fractal. 

In some specific cases (e.g., the Gasket) an explicit eigenform 
can be given.
The first result of existence of an eigenform
on a relatively general class of fractals was given 
by T. Lindstr\o m in \cite{Li}, where 
it is  proved that there exists
an eigenform with all weights equal to $1$ on the nested fractals,
a class of fractals with good properties of symmetry. 
 C. Sabot in \cite{Sa} proved a rather general criterion
 for the existence of an eigenform, and V. Metz in  \cite{Me}  
 improved the results in \cite{Sa}. 
 In fact, he removed an additional requirement
 present in the paper of Sabot and also,
 considered more general  classes of fractals
than those considered in \cite{Sa} and in the 
present paper.

 In this paper, I prove essentially the same
 existence result as that in \cite{Me}, but by a
 completely  different proof. Such a proof 
 on one hand is  in my opinion
  simpler than
 those of Sabot and of Metz, in that avoids almost completely the use
 of Hilbert's projective metric,
on the other it is based on a natural and general principle which
 could lead us to improve this type of result.
 This principle is that a  map from 
 the (non-empty) interior of a  compact and convex 
 $A$ set in $\bre^n$
 into itself has a   fixed point if it is {\it anti-attracting}. 
 This notion is discussed in Section 3.
 To illustrate it, 
 it is well-known that a map $\phi$ that 
 for every $\bar x\in\partial A$ sends a suiatble
   neighborhood of $\bar x$ toward the interior of $A$
   has a fixed point. Call such a map repulsing, in that
   every point of the boundary is repulsing.
We say that the map is anti-attracting if more generally
for every $\bar x\in\partial A$ it sends a suitable
   neighborhood of $\bar x$ in a direction which is not
   opposite to a given element chosen in the interior of $A$
  independent of $\bar x$.
 The previous, of course are informal definitions.
 For the precise definitions see Section 3.

\section{Notation}
In this Section, we  introduce the notation, based on that  of \cite{Pe1}.
This type of construction was firstly considered in \cite{HHW}. 
A notion similar to that of a fractal triple is discussed in \cite{Kigp}, 
Appendix A,
and called an \textit{ancestor}.  

First, we define the general  fractal setting. The
basic notion is that of   \textit{fractal triple}.
By this, we mean  a triple $\big(V^{(0)},V^{(1)},\Psi\big)$, 
where $V^{(0)}$ is a finite
set with $N\ge 2$ elements, $V^{(1)}$ is a finite set and
 $\Psi$ is a finite set of one-to-one maps from $V^{(0)}$ 
 into $V^{(1)}$ satisfying $V^{(1)}=\bigcup_{\psi\in \Psi} \psi\big(V^{(0)}\big)$. 
  Put 
$$V^{(0)}=\{P_1,...,P_N\}.$$
%
We require that
\begin{itemize}
\item[a)] for each $1\leq j\leq N$, there exists a (unique) function 
$\psi_j\in\Psi$ such that $\psi_j(P_j)=P_j$, and $\Psi=\{\psi_1,...,\psi_k\}$, with $k\geq N$;
\item[b)] $P_j\not\in\psi_i\big(V^{(0)}\big)$ when $i\neq j$ (in other words, 
if $\psi_i(P_h)=P_j$ with $i\in\{1,...,k\}$, $j,h\in\{1,...,N\}$, then $i=j=h$);
\item[c)] all pairs of points in $V^{(1)}$ can be connected by 
a path every edge of which is contained in a set of the form
$\psi_i(V^{(0)})$, in other words
for every $Q,Q'\in V^{(1)}$ there exists a sequence of points
$Q_0,...,Q_n\in\von$ such that $Q_0=Q$, $Q_n=Q'$ and for every $h=1,...,n$
there exists $i_h=1,...,k$ such that $Q_{h-1}, Q_h\in \psi_{i_h} (V^{(0)})$.
\end{itemize}

Note that $V^{(0)}\subseteq V^{(1)}$. 
As discussed in Introduction, $V^{(0)}$ is seen as a sort of boundary of
the fractal. 
By definition, a  {\it 1-cell} (or simply a {\it cell}) is a  set of the form 
 $V_i:=\psi_i\big(V^{(0)}\big)$
with $i=1,...,k$.
%
%
The points $P_j$, $j=1,...,N$ will be called {\it vertices},
Let
$$J=J\big(V^{(0)}\big)=\big\{\{j_1,j_2\}:j_1, j_2\in \{1,...,N\}, j_1\neq j_2\big\}.$$

Based on a fractal triple, we can construct in a standard way
a (unique) finitely ramified fractal, more precisely
a P.C.F. self-similar set.  See, for example,
\cite{Kigb}, Appendix A, for the details of such a construction.

Next, we define the Dirichlet forms on $\vo$,
invariant  with respect to an additive constant.
Namely,
denote by $\mathcal{D}\big(V^{(0)}\big)$, or simply $\mathcal{D}$, 
 the set 
of functionals $E$ from $\mathbb{R}^{V^{(0)}}$ into $\mathbb{R}$ of the form
$$E(u)=\sum_{\{j_1,j_2\}\in J} E_{\{j_1,j_2\}}\big(u(P_{j_1})-u(P_{j_2})\big)^2,$$
where $E_{\{j_1,j_2\}} \geq 0$. The numbers $E_{\{j_1,j_2\}}$ 
will be called \textit {coefficients} of $E$.
 Denote 
by $\widetilde{\mathcal{D}}\big(V^{(0)}\big)$, or simply $\widetilde{\mathcal{D}}$, 
the set of the irreducible Dirichlet forms, i.e. 
$$\widetilde{\mathcal{D}}=\{E\in \mathcal{D}:E(u)=0 
\text{ if and only if } u \text{ is constant}\}.$$
We remark that, in particular, if $E\in\cd$ and all coefficients of $E$ are strictly positive,
then $E\in \td$.
However, there are forms in $\td$ that have some coefficients equal to $0$.
More precisely, if $E\in\cd$, then $E\in\td$
 if and only if the graph on $\vo $ whose edges are the pairs $\{P_{j_1}, P_{j_2}\}$
 with $E_{\{j_1,j_2\}}>0$ is connected. This means that
 for every $P_{j}, P_{j'}\in\vo$ there exists
 a sequence $j_0=j,j_1,...,j_n=j'$ such that
 $E_{\{j_{h-1}, j_h\}}>0$ for every $h=1,...,n$.
 
Note that  a form $E\in\cd$ is uniquely determined by its coefficients.
Thus, we can identify $E\in\cd$ with the set of its coefficients
$E_{\{j_1,j_2\}}$ in $\bre^J$. In fact,
$$E_{\{j_1,j_2\}}=
{1\over 4}\Big( E\big(\chi_{\{P_{j_1}\}}-\chi_{\{P_{j_2\}}}\big )
- E\big(\chi_{\{P_{j_1\}}}+\chi_{\{P_{j_2}\}}\big) \Big).
$$
 Accordingly, we will equip $\cd$ with the euclidean metric in $\bre^J$. 
We will also  use the following convention:
$$E\le E' \iff E(u)\le E'(u)\quad \forall\, u\in\rvo,$$
$$E\preceq E'\iff E_d\le E'_d\quad \forall d\in J.$$
Note that $E\preceq E'$ implies $E\le E'$ but the converse does
not hold. 
The following lemma is standard. I merely sketch the proof.

\begin{lemma}\label{2.1}
If $E_1,E_2\in\td$ there exist positive constant $c, c'$ such that
$c E_1\le E_2\le c' E_1$.
\end{lemma}
\begin{proof}
Let $S:=\big\{u\in\rvo: u(P_1)=0, ||u||=1\big\}$.
The ratio ${E_2\over E_1}$ attains its minimum $c$ and its maximum $c'$ over $S$.
Thus, for every non-constant $u\in\rvo$, putting
$\disp{\witi u:={u-u(P_1)\over ||u-u(P_1)||}}$, we have $\witi u\in S$, thus
$${E_2(u)\over E_1(u)} = {E_2(\witi u)
 \over E_1(\witi u)}\in [c,c'].$$
 \end{proof}

Next, we recall the definition of the \textit{renormalization operator} $\Lambda_r$. 
For every $r\in W:=]0,+\infty[^k$, every $E\in \widetilde{\mathcal{D}}$ and
every $v\in \mathbb{R}^{V^{(1)}}$, define
$$S_{1,r}(E)(v)=\sum_{i=1}^{k}r_i \, E(v\circ \psi_i).
$$
Here, an  element $r$ of $W$ can be written as $(r_1,...,r_k)$
and the number $r_i>0$ is called the {\it weight} placed on the cell
$V_i$. 
Note that $S_{1,r}(E)$ is a sort of sum of $E$ on all cells. It is easy to see
that $S_{1,r}(E)$ is a Dirichlet form on $\bre^{\von}$.
Now, for $u\in \mathbb{R}^{V^{(0)}}$ let
$$\mathcal{L}(u)=\left\{v\in \mathbb{R}^{V^{(1)}}:v=u \text{ on } V^{(0)}\right\},$$
and let us define $\Lambda_r(E)(u)$ for $u\in\bre^{\vo}$ as  
$$\Lambda_r(E)(u)=\inf
\left\{S_{1,r}(E)(v):v\in \mathcal{L}(u)\right\}.$$
The form $\Lambda_r(E)$ is called the \textit{restriction} of
$S_{1,r}(E)$ on $\vo$.
Note that $\Lambda_r$ maps  $\widetilde{\mathcal{D}}$ into itself.
For details see Lemma 2.3.5 in \cite{Kigb}.

 If $r\in W$, we say that $E\in \widetilde{\mathcal{D}}$ is an \textit{$r$--eigenform}
 (with eigenvalue $\rho$)
 if there exists $\rho>0$  such that
$$\Lambda_r(E)=\rho \, E,\eqno(2.1).$$
We say that $E$ is an $r$-degenerate eigenform 
 (with eigenvalue $\rho$) if $E\in\td\setminus\cd$
satisfies $(2.1)$.

\begin{lemma}\label{2.2}
The map 
$(r,E)\mapsto \Lb_r(E)$ from $W\times \cd$ to $\cd$ is continuous.
\end{lemma}

The aim of this paper will be to give sufficient conditions for the existence of
an $r$-eigenform.

\section{The fixed point Theorems.}

In this Section, we give two fixed point Theorems, useful for the following.
They are simple variant of the Brouwer fixed point Theorem.
The first concerns maps from a convex and compact set 
{\it not necessarily into itself}
but such that any point  $x$ of the boundary is mapped not 
on the half-line with  end-point at $x$ and opposite to
a given interior point. The second theorem is a
 variant of the first but for {\it open} convex sets.
 First, recall some notation. An affine subset of $\bre^n$ is a set in $\bre^n$ of the form 
 $X+a$ where $X$ is a linear subspace of $\bre^n$ and $a\in\bre^n$.
 Now, let $Z$ be an affine set in $\bre^n$, and let $v,w\in Z$.
 Let
 $$
 ]v,w[:=\{ v+t(w-v): t\in ]0,1[\}, $$
 $$[v,w[:=]v,w[\cup \{v\},\quad  ]v,w]:=]v,w[\cup \{w\},\quad 
 [v,w]:=]v,w[\cup \{v,w\}.
 $$
In the following, if $Z$ is an affine subset
of $\bre^n$, every topological notion on $Z$ will be meant to
be with respect to the topology on $Z$ inherited by the euclidean topology
on $\bre^n$. For example, if
$A\subseteq Z$, we will denote by $\innt(A)$ the interior
of $A$ with respect to such a topology.
 The following lemma is standard and can be easily proved.

 \begin{lemma}\label{3.0}
Suppose $Z$ is an affine subset of $\bre^n$ and $A$ 
is a convex subset of $Z$.Then

i) If  $v\in A$ and $w\in \innt(A)$, then $]v,w[ \, \subseteq \innt(A)$.

ii) $\innt(A)$ is convex.

 \end{lemma}

Let  $\witi x$ be a point of the affine subset $Z$ of $\bre^n$. Then, we define
$$
Ext_{\witi x}(x)=\big\{\witi x+t(x-\witi x): t>1\big\},
$$
and for short we will write $Ext(x)$ instead of $Ext_{\witi x}(x)$ 
when $\witi x$ is clear from the context.

\begin{lemma}\label{3.00}
Let $K$ be a compact convex subset of the affine subset $Z$ of $\bre^n$,
 and let $\witi x\in\innt (K)$. Then

i) For every $x\in Z \setminus \{ \witi x\}$ there exists a unique
$y=p(x)\in\partial K$  of the form $y=\witi x+t(x-\witi x)$, $t>0$.

ii) The map $p: Z \setminus \{\witi x\}\to\partial K$ is continuous.
\smallskip

iii) 
If $x\in\partial K$, then $p(x)=x$;

if $x\in K\setminus\{\witi x\}$, then $x\in[p(x), \witi x]$;

if $x\in Z\setminus \innt (K)$, then $p(x)\in [x,\witi x]$.

\smallskip
iv) If $x\in K \setminus\{\witi x\}$ and $x_1  \in Ext(x)\cap K$,  then 
$x_1\in [x, p(x)]$.

v) If $x\in Z\setminus K$, then
$y=p(x)$ is the unique point in $\partial K$ satisfying
$x\in Ext(y)\cup\{y\}$.
\end{lemma}
\begin{proof} (Sketch)
i) Clearly,
$H:=\partial K\cap\{\witi x+t(x-\witi x):t>0\}$
is non-empty by connectedness.
In fact, the point  $\witi x+t(x-\witi x)$ belongs to
$\innt (K)$ for $t=0$, and,
in view of the boundness of $K$ lies in $Z\setminus K$ for
sufficiently large $t$. 
Moreover, $H$ cannot contain two different points by
Lemma \ref{3.0}. ii) The continuity of $p$ follows at once from
the uniqueness of the point defining $p(y)$.
iii) It is easy to see that in the definition
of $p$ we have $t=1$ if $x\in\partial K$,
$t\ge 1$ if $x\in K\setminus\{\witi x\}$,
$t\le 1$ if $x\in Z\setminus \innt(K)$.
iv)  and v) are trivial.
\end{proof}

\begin{teorema}\label{3.1}
Let $Z$ be an affine
subset of $\bre^n$, let $K$ be a compact convex subset of $Z$,
 and let $\witi x$ be an interior 
point of $K$. Let $\phi:K\to Z$ be a continuous map such that
for every $x\in \partial K$ $\phi(x)\notin Ext(x)$.
  Then $\phi$ has  a fixed point
on $K$.
\end{teorema}

\begin{proof}
Let $\witi p(y)=\begin{cases}
p(y) & $if$\ \ y\in Z\setminus K\\
y  & $if \ $y\in K
\end{cases}$. 
Since $p= Id$ on $\partial K$, then $\witi p$ is continuous on all of $Z$
 with values in $K$ and amounts to $Id$ on $K$.
Let $\witi\phi=\witi p\circ\phi$. Since $\witi \phi$ is continuous from $K$ into itself,
 it has a fixed point $\bar x$. We claim that $\phi(\bar x)=\bar x$.
In fact, if $\bar x\in \partial K$ and 
 $\phi(\bar x)\ne \bar x=\witi p\big( \phi(\bar x)\big)$,  by the definition
of $\witi p$  we have $\bar x= p\big( \phi(\bar x)\big)$
and $\phi(\bar x)\in Z \setminus K$. Thus,
by  Lemma \ref{3.00} v)
 we have $\phi(\bar x)\in Ext(\bar x)$, contrary to our assumption.
If $\bar x\in int (K)$, then $\bar x=\witi p\big( \phi(\bar x)\big)$. Thus,
since $\witi p$ sends $Z\setminus K$ into $\partial K$, we have
$\phi(\bar x)\in  K$, and therefore $\bar x=\witi p\big( \phi(\bar x)\big)=\phi(\bar x)$.
\end{proof}

Let $Z$ and $K$ be as in Theorem \ref{3.1}.
Let now $\phi$ be a  continuous map
from $\innt(K)$ into itself. We say that 
$\bar x\in\partial K$ is {\it anti-attracting
for $\phi$}
if there exists a neighborhood $U_{\bar x}$ of $\bar x$ 
in $Z$ such
that for every $x\in U_{\bar x}\cap \innt (K)$ we have
$\phi(x)\notin Ext(x)$. We say that $\phi$ is anti-attracting if
every pointy of $\partial K$ is anti-attracting for $\phi$.

\begin{teorema}\label{3.2}
Let  $Z$, $K$ and $\witi x$ be as  Theorem \ref{3.1}. Let $\phi$ be an
anti-attracting
map from $\innt(K)$ into itself. Then $\phi$ has a fixed 
point on $\innt (K)$.
\end{teorema}
\begin{proof}
For every $\bar x\in\partial K$,  let $U_{\bar x}$ be a  neighborhood of 
$\bar x$ in $Z$ as in the definition of
an anti-attracting point.
We can and do assume that, 
$U_{\bar x}$ is open and moreover its closure has the same property,
namely
$$\Big(x\in \bar{ U_{\bar x}}\cap \innt (K)\Big)\Rightarrow
\Big(x\ne\witi x, \phi(x)\notin  Ext(x)\Big) \eqno(3.1)
$$
By compactness, there exist $\bar x_1,...,\bar x_m\in\partial K$ such that
$$U:=\bigcup\limits_{i=1}^m U_{\bar x_i}\supseteq \partial K.$$
Let
$\witi K:=co\big(  K\setminus U\big) $.
Note that, in view of Lemma \ref{3.0} ii), we have
 $$\witi K\subseteq  \innt (K).\eqno (3.2)$$
We also have
$$
\partial {\witi K}\subseteq \bar U. \eqno (3.3)$$
In fact, in the opposite case,
there exists $x\in \partial {\witi K}$,  such that 
$$x\in 
\innt(K)\setminus \bar U\subseteq K\setminus U\subseteq \witi K,$$
and since $\innt(K)\setminus \bar U$ is open in $Z$,
then $x\notin \partial{\witi  K}$, a  contradiction,
thus (3.3) holds. By (3.2) and (3.3),
for every $x\in\partial {\witi  K}$, we have 
$x\in \bar{ U_{\bar x_i}}\cap \innt (K)$
for some $i=1,...,m$, thus by (3.1) $\phi(x)\notin  Ext(x)$.
Moreover, $\witi x\notin \bar U$ by (3.1). Therefore,
$\witi x\in K\setminus U\subseteq \witi K$, but
in view of (3.3), $\witi x\notin\partial \witi K$, thus
$\witi x\in\innt (\witi K)$.
The map $\phi$ from $\witi K$ into $Z$ thus satisfies all hypotheses of
 Theorem \ref{3.1},  thus
$\phi$ has a fixed point on $\witi K \subseteq \innt (K)$.
\end{proof}

\section{Anti-attracting forms  on Fractals.}

In this Section, 
we investigate  the notions of Section 3 
in the setting of forms in $\cd$. Namely, we define     specifics sets
in $\bre^J$ which will play the role of $Z$ and $K$
in Section 3. Moreover, we will investigate the notion
of an anti-attracting form with respect to a map obtained
normalizing $\Lb_r$. Let 
$$|x| :=\sum\limits_{d \in J} x_d \quad \forall\, x\in\bre^J,$$
$$Z:=\big\{E\in\bre^J: |E|=1\big\},$$
$$\cdn =:\big\{E\in\cd: |E|=1\big\}=\big\{E\in Z: E_d\ge 0\quad\forall\, d\in J\big\}.$$
So, $Z$ is an affine set in $\bre^J$ and $\cdn$ is a compact and convex subset
of $Z$. 
Note that 
$$\max\limits_{d\in J} E_d\ge \witi m:={1\over \#(J)}
\quad\forall\, E\in Z. \eqno (2.1)$$
We easily characterize   $\innt (\cdn)$. In fact we have
$\innt(\cdn)=\cdni{1}$  where
$$\cdni{1}:=\{E\in\cdn: E_d>0\ \forall\, d\in J\}\subseteq \td.$$
We next want to study the  map
$\Lbt_r $
 defined  as 
$$\Lbt_{ r}(E):=\disp{ {\Lb_{r}(E)\over \big| \Lb_{r}(E)\big|} }.$$
As it is known that if $E\in\td$ 
satisfies $E_d>0$ for every $d\in J$, so does $\Lb_r(E)$, then
$\Lbt_r$   maps continuously
$\cdni{1}$ into itself.
However, in general $\Lbt_r$ cannot be extended
continuously on all of $\tdn$. In fact, we could have
$\Lb_r(E)=0$ for some $E\in\cd\setminus\td$.
We so need a nice decomposition of $\partial\cdn$.
Let 
$$\cdni{2}:=\cdn\cap\td\setminus\cdni{1},$$
$$\cdni{3}=\{E\in\cdn\setminus\td: \Lb_r(E)\ne 0\},$$
$$\cdni{4}=\{E\in\cdn\setminus\td:  \Lb_r(E)=0\},$$
where $r\in W$. In fact, 
it can be proved that the  formula
$\Lb_r(E)=0$ is independent of $r\in W$,
but this is not important for our considerations since we fix
a given $r\in W$.
We easily have
$$ \partial\cdn=\cdni{2}\cup \cdni{3}\cup \cdni{4}.$$
We easily see that 
$\Lbt_r$   maps continuously $\cdni{1}\cup \cdni{2}\cup \cdni{3}$
into $\cdn$.
Note also that, when $E\in\cdni{1}\cup\cdni{2}\cup \cdni{3}$,
then $E$ is a (possibly degenerate) $r$-eigenform if and only if
it is a  fixed point of $\Lbt_r$.

We are going to prove
that every point $E\in \partial\cdn$ which is not an
$r$-degenerate eigenform is anti-attracting for $\Lbt_r$. We need the following 
lemma, which is well-known, but however, I will prove it.

\begin{lemma}\label{4.1}
If $E,E'\in\td$ and $0<\rho<\rho'$ we cannot have
$\Lb_r(E)\le \rho E$ and $\Lb_r(E')\ge \rho' E'$.
\end{lemma}
\begin{proof}
By contradiction, if  $\Lb_r(E)\le \rho E$ and $\Lb_r(E')\ge \rho' E'$,
using an inductive argument we obtain
$${\Lb_r}^n(E)\le \rho^n E, \quad {\Lb_r}^n(E')\ge (\rho')^n E'\ \ \forall n\in\bna.$$
However, in view of Lemma \ref{2.1}, there exist $c>0$ $c'>0$
such that $c E'\ge E\ge c' E'$ which implies,
${\Lb_r}^n(E)\ge c ' {\Lb_r}^n(E')$ for every positive integer $n$, 
so that
$$  \rho^n E \ge  {\Lb_r}^n(E)\ge c' {\Lb_r}^n(E')\ge 
c'  {\rho'}^n E'\ge {c'\over c} {\rho'}^n E
$$
and, since  $0<\rho<\rho'$, this cannot hold  for large $n$.
\end{proof}

We now fix a form
$\witi E\in \innt(\cdn)=\cdni{1}$, and,  according
to the notation of the previous section, put
$$Ext(E)=\big\{\witi E+t(E-\witi E): t>1\big\} \ \forall\, E\in Z\setminus\{\witi E\}.
$$
Here, $\cdn$ plays the role of $K$ in Section 3.
Also, define $p:Z\setminus \{\witi E\}\to\partial  \cdn$ as
in the previous section.

\begin{lemma}\label{4.2}
Let $r\in W$. Then every $\bar E\in \cdni{4}$ is
anti-attracting for $\Lbt_r$.
\end{lemma}
\begin{proof}
We prove that there exists a neighborhood
$U$ of $\bar E$ such that 
$$p(E)_d\le 2 E_d \quad \forall\, E\in U\cap\cdn
\quad \forall\, d\in J.\eqno (4.2)$$
Note that by Lemma \ref{3.00} iii), if
$E\in\cdn\setminus\{\witi E\}$,  we have $E\in [p(E),\witi E]$, thus
$$E_d\in [p(E)_d,\witi E_d] \quad  \forall\, d\in J.\eqno (4.3)$$
Thus, since $E\in\cdni{1}$ and by (4.3), 
we have $p(E)_d<E_d$,
and a fortiori   (4.2), for $d$ such that $\bar E_d=0$
(for a suitable $U$). On the other hand, we have
$$E_d\llim_{E\to \bar E} \bar E_d, \quad
p( E)_d\llim_{E\to \bar E} p(\bar E)_d \quad\forall \,d\in J,$$
by the continuity of $p$. 
Since $\bar E\in\cdni{4}\subseteq \partial\cdn$,
by Lemma \ref{3.00} iii)  we have
 $p(\bar E)=\bar E$. Thus
(4.2) holds for a suitable $U$, also for $d$ such that $\bar E_d>0$, and
(4.2) is proved.  
We now prove by contradiction that,
possibly restricting $U$,
given $E  \in U\cap\cdni{1}$ we have 
$$\Lbt_r(E)\notin Ext(E).\eqno (4.4)$$
By Lemma \ref{3.00} iv) we have
$E'_d\in [E_d, p(E)_d]$ for every
 $E'\in Ext(E)\cap \cdn$ and every $d\in J$. Thus, by (4.2),
  if (4.4) does not hold we have
 $$\Lbt_r(E)_d\le 2E_d\quad \forall\, d\in J.\eqno (4.5)$$
Take a positive $\veps$ which we will specify later.
Since $\bar E\in\cdni{4}$, by definition we can choose $U$
such that
$$\big(\Lb_r(E)\big)_d <\veps\quad\forall\, 
E\in U\cap \cdn,\quad \forall\,  d\in J.\eqno(4.6)$$
For such $E$, by the definition of
$\Lbt_r$ we have 
$\alpha \Lbt_r(E)=
\Lb_r(E)$ for some $\alpha>0$. Thus, by (4.1),  for some $\bar d\in J$ we have
$$\alpha \witi m\le \alpha \big(\Lbt_r(E)\big)_{\bar d}=
\big(\Lb_r(E)\big)_{\bar d}<\veps.
$$
It follows that 
 $\alpha<{\veps\over\witi m}.$
Hence, in view of (4.5)
we have
$$\big(\Lb_r(E)\big)_d=\alpha\big(\Lbt _r(E)\big)_d\le
2{\veps\over \witi m} E_d\quad\forall\, d\in J.
$$
Thus,  we have
$\Lb_r(E)\le \disp{{2\veps\over \witi m}E}$.
If $c>0$ is so that $\Lb_r(\witi E)\ge c \witi E$ (see Lemma \ref{2.1})
and we choose $\veps$ so that $\disp{{2\veps\over \witi m}<c}$, we have contradicted
Lemma \ref{4.1}. Such a  contradiction
shows that (4.4) holds and the Lemma is proved.
\end{proof}

\begin{lemma}\label{4.3}
Let $r\in W$. Then every $\bar E  \in \cdni{2}\cup\cdni{3}$ 
such that  $\Lbt_r(\bar E)\ne \bar E$
is anti-attracting for $\Lbt_r$.
\end{lemma}
\begin{proof}
Since $p(\bar E)=\bar E$
 we have $\Lbt_r(\bar E)\notin [\bar E, p(\bar E)]=\{\bar E\}$.
By continuity, there exists $U$ neighborhood of $\bar E$
such that for every $E\in U\cap \cdni{1}$ we have
$\Lbt_r(E)\notin [E, p( E)]$, thus, by Lemma \ref{3.00} iv), 
$\Lbt_r(E)\notin Ext(E)$.
\end{proof}

\section {The Theorem.}

In view of Lemmas \ref{4.2} and \ref{4.3}, we can use
Theorem \ref{3.2}, provided that also every degenerate $r$-eigenform
in $\partial \cdn$ is anti-attracting.
However, this does not necessarily hold, but depends 
on the $r$-eigenform. More precisely, we have
to study carefully the local behavior of $\Lb_r$ near a
degenerate  $r$-eigenform. Recall that if $\bar E\in \cd\setminus\td$,
$\bar E\ne 0$,
then the {\it kernel} $\bar E^{-1}(0)$ of $\bar E$ 
strictly contains the set of the constant functions.

Now, an argument due to Sabot (see \cite{Sa}) shows that
we can approximate $\Lb_r$ near $\bar E$ by minimizing
along functions in $\bar E^{-1}(0)$.
Namely, for $u\in \bar E^{-1}(0)$ let 
$$\Lb_{r, \bar E}(E)(u)=\inf\big\{S_{1,r}\big(E)(v):
v\in \cl(u), v\circ\psi_i\in \bar E^{-1}(0)\ \forall\, i=1,...,k\big\}.
$$
We say that the degenerate $r$-eigenform 
 $\bar E \in\cdni{3}$ with eigenvalue $\rho$ is {\it repulsing} if 
$$
\exists\, \rho'\ge\rho: \,  \Lb_{r, \bar E}(\witi E)(u)\ge \rho'  \witi E(u)
  \ \forall\,   u\in \bar E^{-1}(0). 
\eqno (5.1)$$
We now prove that every  repulsing degenerate $r$-eigenform
 $\bar E\in\cdni{3}$ is anti-attracting for $\Lbt_r$.
 We need three preliminary lemmas.

\begin{lemma}\label{5.2}
If $E\in\cdni{1}$,  the ratio $\disp{E\over \tilde E}$ attains
 a maximum $\eta_E$ on the set of all non-constant functions
 in $\bar E^{-1}(0)$.
 \end{lemma}
\begin{proof}
The proof is similar to that of Lemma \ref{2.1}.
\end{proof}

\begin{lemma} \label{5.3} We have

i) $\eta_E>0$.

ii) $E(u)\ge \eta_E \witi E(u)\quad\forall\, u\in \bar E^{-1}(0)$.

iii) The exists a non-constant $\bar u\in \bar E^{-1}(0)$ such that
$E(\bar u)=\eta_E \witi E(\bar u)$.
\end{lemma}
\begin{proof} Trivial. \end{proof}

 The following Lemma is the most technical point in this paper,
 where we use a previous result of Sabot (a similar result was proved later
 by Metz in \cite{Me}) whose proof is long.

\begin{lemma}\label{5.4}
Let $r\in W$. If $\bar E\in\cdni{3}$ and $\Lbt_r(\bar E)=\bar E$,
then for every $\alpha<1$ there exists
$U$  neighborhood
of $\bar E$ such that 
$$\Lb_r(E)(u)\ge \alpha\eta_E \Lb_{r,\bar E}(\witi E)(u)
\quad  \forall \, E\in U\cap \cdni{1}\ \forall u\in \bar E^{-1}(0).$$
\end{lemma}
\begin{proof}
This is a consequence of the arguments in \cite{Sa}.
For example, by  \cite{Sa},  Prop. 4.23 
(see also Prop. 23 in \cite{Me}) there exists $U$  neighborhood
of $\bar E$ such that  for every $E\in U\cap \cdni{1}$ and every
 $u\in \bar E^{-1}(0)$ we have 
 $\Lb_r(E)(u)\ge \alpha  \Lb_{r,\bar E}(E)(u)$. The Lemma follows from
 Lemma \ref{5.3} iii) and the definition of $\Lb_{r,\bar E}(E)$.
\end{proof}

\begin{lemma}\label{5.5}
Let $r\in W$.
Then every  repulsing degenerate $r$-eigenform
 $\bar E\in\cdni{3}$ is anti-attracting for $\Lbt_r$.
\end{lemma}
\begin{proof}
By Lemma 5.4 and (5.1), given $\alpha\in]0,1[$,
we find a neighborhood $U$ of $\bar E$ such that,
if $E\in U\cap\cdni{1}$ and $u\in\bar E^{-1}(0)$ we have
$$ \Lb_r(E)(u)\ge\rho'\alpha\eta_E \witi E(u). \eqno (5.2)$$
Next, note that $\Lb_r  (\bar E)=\rho\bar E$,  
hence $|\Lb_r  (\bar E)|=\rho|\bar E|=\rho$.
Thus, for $u\in \bar E^{-1}(0)$ we have
$$  \Lbt_r(E)(u)$$
$$ ={ \Lb_r  (E)(u)\over |\Lb_r  (\bar E)|}  {|\Lb_r  (\bar E)|\over |\Lb_r (E)|}$$
$$\ge {\rho'\over\rho}    \alpha  \eta_E {|\Lb_r  
(\bar E)|\over |\Lb_r (E)|} \witi E (u).$$
Also, possibly restricting $U$, by the continuity 
of $\Lb_r$ we can assume that
$|\Lb_r(\bar E)|\ge\alpha   |\Lb_r(E)|$. Hence
$ \Lbt_r(E)(u)\ge \alpha^2 {\rho'\over\rho}
\eta_E   \witi E (u)$. Since $\rho'>\rho$ we can choose $\alpha$ such that
$\alpha^2{\rho'\over\rho}>1$. Thus if $u\in\bar E^{-1}(0)$ is non-constant we have
$$\Lbt_r(E)(u)>\eta_E \witi E(u).\eqno (5.3)$$
It follows that $\bar E$ is anti-attracting for $\Lbt_r$ since
$\Lbt_r(E)\notin Ext(E)$ for every $E\in U\cap\cdni{1}$.
In fact, in the opposite case,
  by Lemma \ref{3.00} iv), we have 
 $$\Lbt_r(E)\in[E,p(E)].\eqno (5.4)$$
  On the other hand, by Lemma \ref{3.00} iii) we have
  $$E\in [p(E),\witi E].\eqno (5.5)$$
  Next, note that, since $\bar E(\bar u)=0<\witi E(\bar u)$, possibly restricting
  $U$, we can assume that $E(\bar u)<\witi E(\bar u)$. Thus,
  by (5.5) we have $p(E)(\bar u)\le E(\bar u)$, thus, by (5.4) and Lemma
  \ref{5.3} iii)
  we have $\Lbt_r(E)(\bar u)\le E(\bar u)=\eta_E \witi E(\bar u)$.
  This contradicts (5.3) and the Lemma is proved. \end{proof}

\begin{teorema}
Suppose that  every degenerate $r$-eigenform in $\cdni{3}$ is repulsing.
Then there exists an $r$-eigenform.
\end{teorema}
\begin{proof}
Suppose there exists no $r$-eigenform, in particular 
$\Lbt_r(E)\ne E$ for every $E\in \cdni{2}$.
By Lemmas 4.2, 4.3 and 5.5,  the hypothesis of
Theorem \ref{3.2} is satisfied with $K=\cdn$, and $\phi=\Lbt_r$.
Thus, there exists $\bar E\in\cdni{1}$ such that
$\Lbt_r(\bar E)=\bar E$, hence $\bar E$ is an $r$-eigenform.
\end{proof}

\end{document}